\documentclass[12pt]{amsart}
\usepackage[margin=1in]{geometry}
\usepackage{mathptmx,amsmath,amsthm,tikz,cite}
\usetikzlibrary{patterns,arrows,decorations.pathreplacing, calc}
\theoremstyle{plain}
\newtheorem*{theorema}{Theorem A}
\newtheorem*{theoremb}{Theorem B}
\newtheorem{theorem}{Theorem}[section]
\newtheorem{proposition}[theorem]{Proposition}

\newtheorem{lemma}[theorem]{Lemma}

\theoremstyle{definition}
\newtheorem{definition}[theorem]{Definition}

\theoremstyle{remark}

\newcommand\Div{\operatorname{Div}}
\newcommand\Pic{\operatorname{Pic}}
\newcommand\Prin{\operatorname{Prin}}
\newcommand\supp{\operatorname{supp}}

\allowdisplaybreaks

\begin{document}

\title{Chip-Firing on Trees of Loops}
\author{Sameer Kailasa}
\email{kailasas@uchicago.edu}
\address{University of Chicago}
\author[Vivian Kuperberg]{Vivian Kuperberg*}\thanks{*Corresponding author}
\email{vzk2@cornell.edu}
\address{Cornell University}
\author{Nicholas Wawrykow}
\email{nicolas.wawrykow@yale.edu}
\address{Yale University}

\begin{abstract}
Cools, Draisma, Payne, and Robeva proved that generic metric graphs that are ``paths of loops'' are Brill-Noether general. We show that Brill-Noether generality does not hold for ``trees of loops'': the only trees of loops that are Brill-Noether general are paths of loops. We study various notions of generality and examine which of these graphs satisfy them.
\end{abstract}

\maketitle

\section{Introduction}

Let $\Gamma$ be a compact tropical curve, or a \emph{metric graph}. The polyhedral subset $W^r_d(\Gamma)$ of $\Pic^d(\Gamma)$, consisting of linear equivalence classes of divisors of degree $d$ and rank at least $r$, has expected dimension
\[\rho(g,r,d) = g - (r+1)(g-d+r). \]

This expectation is formalized by the notion of \emph{Brill-Noether generality}, in which a graph $\Gamma$ is Brill-Noether general if $W^r_d(\Gamma)$ is empty whenever $\rho < 0$ and  $\mathrm{dim}(W^r_d(\Gamma)) = \mathrm{min}\{\rho,g\}$ whenever $\rho \ge 0$. Very few examples of Brill-Noether general graphs are known. The most studied is the set of metric graphs that are combinatorially a path, or chain, of $g$ loops, with generic edge lengths. In \cite{cdpr}, these are shown to be Brill-Noether general. A natural combinatorial generalization is to consider metric graphs $\Gamma$ that are trees of $g$ loops, with generic edge lengths. In this paper we examine these graphs and show that the expectation $\rho$ for the dimension of $W^r_d(\Gamma)$ is never accurate unless $\Gamma$ is a path of loops.

\begin{theorema}
\label{onlypaths}
A tree of loops $\Gamma$ is Brill-Noether general if and only if $\Gamma$ is a path of loops.
\end{theorema}

This negative result inspires the definition of other notions of Brill-Noether generality. One natural approach is to weaken Brill-Noether generality to the boolean condition that $\rho$ is nonnegative if and only if $W^r_d(\Gamma)$ is nonempty. We refer to this boolean condition as \emph{weak Brill-Noether generality}. However, examining this condition for general trees of loops gives another negative result.

\begin{theoremb}
Let $\Gamma$ be a tree of loops of genus $g$ such that the longest path of loops consists of at most $g-2$ loops. Then $\Gamma$ is not weakly Brill-Noether general.
\end{theoremb}

Another approach, introduced in \cite{lpp}, is to examine the \emph{Brill-Noether rank} $w^r_d(\Gamma)$ as opposed to the dimension of the space $W^r_d(\Gamma)$. The Brill-Noether rank also has expected value $\rho$. Similarly, the condition of \emph{rank Brill-Noether generality} is identical to the condition of \emph{Brill-Noether generality}, but measuring $w^r_d(\Gamma)$ (instead of $\mathrm{dim}(W^r_d(\Gamma))$) against $\rho$. The authors in \cite{lpp} show that for $\Gamma$ a loop of loops of genus $4$, $w^1_3(\Gamma) = 0 = \rho(4,1,3)$. As a result, rank Brill-Noether generality may hold in some cases where geometric Brill-Noether generality does not. This paper concludes by introducing a possible technique for examining the rank Brill-Noether generality of trees of loops.

\section{Background}

We will be working within the realm of divisors on metric graphs. A \emph{metric graph} $\Gamma$ is a compact connected metric space such that for all points $p \in \Gamma$, $p$ has a neighborhood isometric to a star-shaped set. A \emph{divisor} $D$ on a metric graph $\Gamma$ is an element of the free abelian group generated by $\Gamma$; we write $\Div(\Gamma)$ for the group of these divisors under addition. For a divisor $D=a_{1}v_{1}+\cdots+a_{n}v_{n}$ with $a_i \in \mathbb Z$ for all $i$, the \emph{degree} of $D$ is the sum of the coefficients $a_{1}+\cdots+a_{n}$, and a divisor is called \emph{effective} if all coefficients $a_i$ are nonnegative. We denote by $D(v)$ the coefficient of $v$ in the divisor $D$, so that 
\[D = \sum_{v \in \Gamma} D(v) \cdot v. \]

For $f$ a continuous piecewise linear function on $\Gamma$ with integer slopes, we may consider the \emph{divisor of $f$} given by
\[\mathrm{div}(f) = \sum_{v \in \Gamma} \mathrm{ord}_v(f) \cdot v \]
where for all $v \in \Gamma$, $\mathrm{ord}_v(f)$ is the sum of the incoming slopes of $f$ at $v$. Divisors of piecewise linear functions are called \emph{principal}, and $\Prin(\Gamma)$ denotes the additive group of principal divisors. 

Two divisors $D$ and $D'$ on a metric graph $\Gamma$ are said to be \emph{equivalent}, \emph{chip-firing equivalent}, or \emph{linearly equivalent}, written $D\sim D'$, if $D - D' \in \Prin(\Gamma)$. Informally, $D$ and $D'$ are equivalent if and only if one can move from $D$ to $D'$ via moves in the \emph{chip-firing game on metric graphs}. In this game, a divisor $D$ is thought of as a configuration of finitely many chips placed on a metric graph, with $D(v)$ the number of chips at any point $v$. Negative coefficients are taken to be piles of ``anti-chips'' instead, where an anti-chip and a chip cancel whenever they collide. One can then ``fire'' a point on the graph with a certain speed. When a point $p$ is fired, a chip is sent along every edge adjacent to $p$ at the same speed, so that each chip lands the same distance away from the point $p$. On a metric graph, any closed subset can in fact be fired, so that 
chips are sent along outgoing edges. The intuition of firing a closed subset exactly aligns with the concept of adding a principal divisor. Chip-firing on metric graphs is also discussed in \cite{bn},\cite{bs}, \cite{cdpr}, and \cite{mz}. In \cite{o}, Osserman examines chip-firing on metric graphs with an approach akin to the one presented here.

For a point $v \in \Gamma$ and a divisor $D$, one can consider the unique \emph{$v$-reduced divisor} $D_0$ equivalent to $D$, which satisfies the following two conditions.
\begin{enumerate}
\item $D_0$ is effective away from $v$.
\item Let $A \subseteq \Gamma \setminus \{v\}$ be any closed connected set. Then there exists $p \in \partial A$ with $\mathrm{outdeg}_A(p) < D(p)$. Here $\mathrm{outdeg}_A(p)$ is the degree of $p$ in the graph $\Gamma \setminus A \cup \{p\}$. Intuitively, this condition means that no more chips may be fired towards $v$ while preserving the effectiveness of $D_0$ away from $v$.
\end{enumerate}
The $v$-reduced divisor $D_0$ can be obtained from $D$ via  \emph{Dhar's burning algorithm}, which is explained in \cite{l}. For each divisor $D$ on a metric graph $\Gamma$, and for each $v \in \Gamma$, there is a unique $v$-reduced divisor equivalent to $D$. Among divisors equivalent to $D$ and effective away from $v$, the unique $v$-reduced divisor has maximal coefficient at $v$. 

The \emph{Picard group} $\Pic^0(\Gamma)$ of $\Gamma$ is the quotient group $\Div^0(\Gamma)/\sim$, which can be viewed as the $g$-dimensional torus $H_1(\Gamma,\mathbb R)/H_1(\Gamma,\mathbb Z) \cong \mathrm{Jac}(\Gamma)$ via the Abel-Jacobi map. Similarly, for any degree $d$ we define $Pic^d(\Gamma) = \Div^d(\Gamma)/\sim$, which is a $\Pic^0(\Gamma)$-torsor for every $d$. Moreover, $\Gamma$ and thus $H_1(\Gamma,\mathbb R)$ is endowed with a natural ``cycle intersection'' bilinear form  (see \cite{mz}, \cite{bf}, \cite{abks}).

The question at hand is then which equivalence classes of divisors contain an effective divisor, and the robustness of this containment. This concept is made rigorous by the definition of a divisor's \emph{rank}.

\begin{definition}
The \emph{rank} $r(D)$ of a divisor $D$ on a metric graph $\Gamma$ is the largest nonnegative integer $r$ such that for every effective divisor $E$ of degree $r$ on $\Gamma$, the divisor $D-E$ is equivalent to an effective divisor. If $D$ is not equivalent to an effective divisor, $r(D)$ is defined to be $-1$.
\end{definition}

One particular divisor, known as the \emph{canonical divisor} $K$ on a metric graph $\Gamma$ is defined as
\[K = \sum_v (\deg v - 2)v,\]
ranging over all vertices $v \in \Gamma$. Then $\deg K = 2g - 2$, which can be checked by examining the Euler characteristic of $\Gamma$. This divisor is a key part of the tropical Riemann-Roch Theorem, which holds for metric graphs and is a useful result in the study of divisors on metric graphs.

\begin{theorem}[Tropical Riemann-Roch Theorem] \cite{mz}, \cite{gk}

Let $D$ be a divisor on a metric graph $\Gamma$ of genus $g$. Then
\[r(D) - r(K-D) = \deg(D) + 1 - g.\]
\end{theorem}

Our main application of the tropical Riemann-Roch Theorem will be the case when $\Gamma$ is a single loop. In that case the divisor $K$ has no chips whatsoever. For any divisor $D$ with positive degree, $K-D$ then has negative degree. Thus the rank $r(K-D) = -1$, so the Riemann-Roch theorem tells us that
\[r(D) +1 = \deg(D) + 1 - 1 = \deg(D).\]
So for the graph $\Gamma$ consisting of a single loop, for any divisor $D$ with $\deg(D) \ge 1$, the rank $r(D)$ is given by $r(D) = \deg(D) - 1$. We will frequently use this fact.

For each degree $d$ and rank $r$, we would like to examine $W^r_d(\Gamma) \subseteq \Pic^d(\Gamma)$, which denotes the set of all divisors of $\Gamma$ that have degree exactly $d$ and rank at least $r$. The set $W^r_d(\Gamma)$ is a polyhedral subset of $\Pic(\Gamma)$, but it is not necessarily pure dimensional (see, for example, \cite{lpp}). It is then natural to ask what its dimension is, where the dimension of $W^r_d(\Gamma)$ is defined as the largest dimension of any cell in $W^r_d(\Gamma)$. If $\Gamma$ has genus $g$, the algebraic-geometric analogue suggests that the dimension of $W^r_d(\Gamma)$ is the Brill-Noether estimate $\rho(g,r,d) = g - (r + 1)(g - d + r)$. We spend the remainder of this paper examining the accuracy of this estimate for metric graphs that are combinatorially trees of loops.

\section{Geometric Brill-Noether generality}

\begin{definition}
A metric graph $\Gamma$ is \emph{geometric Brill-Noether general} if:
\begin{enumerate}
\item $W^r_d(\Gamma)$ is empty whenever $\rho(g,r,d)$ is negative.
\item $W^r_d(\Gamma)$ has dimension $\mathrm{min}\{\rho,g\}$ whenever $\rho(g,r,d)$ is nonnegative.
\end{enumerate}
\end{definition}

The authors in \cite{cdpr} have shown that a general path of loops is geometric Brill-Noether general. This leads to the question of whether or not the same holds for trees of loops, which we address using inequalities. These inequalities, presented in \ref{dim_inequality}, might be of independent interest. However, we will first present a tree of loops of small genus that is not geometric Brill-Noether general.

\begin{lemma}
\label{w13_mismatch}
Let $\Gamma$ be a tree of loops of genus $4$ that is not a path of loops. Then $\Gamma$ is not geometric Brill-Noether general; in particular, $\mathrm{dim}(W^1_3(\Gamma)) > \rho(4,1,3)$.
\end{lemma}
\begin{proof}
Since $\Gamma$ is a tree of loops of genus $4$ that is not a path of loops, $\Gamma$ must be given by the following picture with some edge lengths.

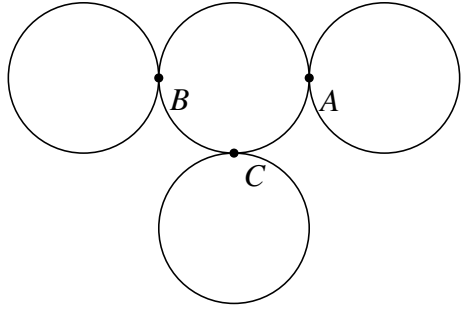
\begin{center}
\begin{tikzpicture}[semithick]
\draw (0,0) circle [radius=1];
\draw [fill] (0,-1) circle [radius=0.05] node[below right] {$C$};
\draw (2,0) circle [radius=1];
\draw [fill] (1,0) circle [radius=0.05] node[below right] {$A$};
\draw (-2,0) circle [radius=1];
\draw [fill] (-1,0) circle [radius=0.05] node[below right] {$B$};
\draw (0,-2) circle [radius=1];
\end{tikzpicture}
\end{center}

For any angle $0 \le \theta < 2\pi$, let $D_\theta$ be the divisor on $\Gamma$ of the form

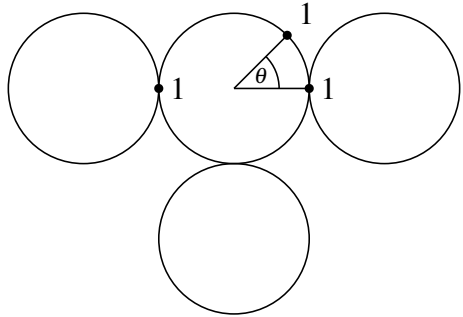
\begin{center}
\begin{tikzpicture}[semithick]
\draw (0,0) circle [radius=1];
\draw (2,0) circle [radius=1];
\draw [fill] (1,0) circle [radius=0.05] node[right] {$1$};
\draw (-2,0) circle [radius=1];
\draw [fill] (-1,0) circle [radius=0.05] node[right] {$1$};
\draw (0,-2) circle [radius=1];
\draw [fill] (0.707106,0.707106) circle [radius=0.05] node[above right] {$1$};
\draw (0,0) -- (0.707106,0.707106);
\draw (0,0) -- (1,0);
\draw (0:.6) arc (0:45:.6);
\draw (25:.4) node {\scalebox{0.7}{$\theta$}};
\end{tikzpicture}
\end{center}

In other words, $\Gamma$ has one chip at each of $A$ and $B$ and a third chip at an angle $\theta$ on the central circle. Each $D_\theta$ has degree $3$. Let $E$ be any divisor of degree $1$; then $E$ consists of one chip at some point $p \in \Gamma$. Crucially, the chip-firing game may be played independently on any loop, by firing all points on one of the other loops whenever a chip is in danger of going from one to the next. Then if $p$ is a point in the central loop, $D_\theta - E = D_\theta - (p)$ has degree $2$ when restricted to the central loop, which by Riemann-Roch is linearly equivalent to an effective divisor. If $D_\theta$ is $A$-, $B$-, or $C$-reduced, it has at least two chips at $A$, $B$, or $C$, respectively. Thus if $p$ is in the right, left, or bottom loop, we may take the $A$-, $B$- or $C$-reduced divisor equivalent to $D_\theta$, respectively. The degree of this divisor with one chip removed from $p$, and restricted to the right, left, or bottom loop, is $1$. By Riemann-Roch this loop is linearly equivalent to an effective divisor, so $D_\theta-E$ must be linearly equivalent to an effective divisor.

In particular, for all $\theta$, $D_\theta$ has rank at least $1$. Thus we have a one-dimensional subset of $W^1_3(\Gamma)$, with parameter $\theta$, so $\mathrm{dim}(W^1_3(\Gamma)) \ge 1$. However, the genus of $\Gamma$ is $4$, and $\rho(4,1,3) = 0 < 1$.
\end{proof}

\begin{definition} Suppose $\Gamma_1$ and $\Gamma_2$ are metric graphs. For any $p \in \Gamma_1$ and $q\in \Gamma_2$, the \textit{wedge sum} $\Gamma_1 \wedge_{p,q} \Gamma_2$ is the metric graph obtained by gluing $\Gamma_1$ and $\Gamma_2$ via the identification $p\sim q$. 

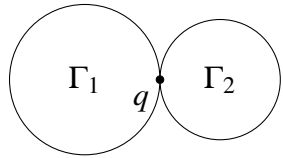
\begin{center}
\begin{tikzpicture}
\draw (-1,0) circle [radius=1] node {$\Gamma_1$};
\draw (0.8,0) circle [radius=0.8] node {$\Gamma_2$};
\draw [fill] (0,0) circle [radius=0.05] node[below left] {$q$};
\end{tikzpicture}
\end{center}
\end{definition}

As mentioned in the proof of Lemma \ref{w13_mismatch}, chip-firing may be performed independently on either side of a wedge point. It is therefore possible to examine the relationship between different values of $W^r_d$ for a wedge sum and its summands. For $\Gamma$ a metric graph and $\Gamma_1 \subset \Gamma$ a metric subgraph, for any divisor $D \in \text{Div}(\Gamma)$, we will denote by $D \vert_{\Gamma_1}$ the \emph{restriction} of $D$ to $\Gamma_1$, namely the divisor $D_1$ on $\Gamma_1$ with $D_1(p) = D(p)$ for all $p \in \Gamma_1$.

\begin{theorem} 
\label{dim_inequality}
Let $\Gamma_1$ be any metric graph, $\Gamma_2$ a loop, and $\Gamma = \Gamma_1 \wedge \Gamma_2$ an arbitrary wedge sum with wedge point $q$. If $W^r_{d} (\Gamma_1)$ is nonempty, then $$\dim(W^{r}_{d+ 1} (\Gamma)) \ge \dim(W^{r}_{d} (\Gamma_1)) + 1$$
\end{theorem}

\begin{proof} 
Let $D \in W^r_d (\Gamma_1)$ and $p\in \Gamma_2$ be arbitrary; we may assume, without loss of generality, that $D$ is $q$-reduced. We claim $r(D + (p)) \ge r$ as a divisor on $\Gamma$. To see this, let $E \in \Div(\Gamma)$ be any effective divisor of degree $r$, which is $q$-reduced without loss of generality, and denote $A:= D+(p) - E$. Let $A_1$ and $A_2$ be the unique divisors on $\Gamma$ such that:

\begin{itemize}
\item $A= A_1 + A_2$, 
\item $\supp(A_i) \subset \Gamma_i$ for $i=1,2$, 
\item $\deg(A_1) = d-r$ and $\deg(A_2) = 1$
\end{itemize} 

Since chip firing moves can be conducted independently on $\Gamma_1$ and $\Gamma_2$, to show that $A$ is equivalent to an effective divisor, we need only show that $A_1$ and $A_2$ are equivalent to effective divisors on $\Gamma_1$ and $\Gamma_2$ respectively. By construction, $A_1 = D - E_1$ for some effective divisor $E_1$ of degree $r$ with $\supp(E_1) \subset \Gamma_1$. Because $D$ has rank $r$ in $\Gamma_1$, it follows $A_1$ is equivalent to an effective divisor on $\Gamma_1$. The Riemann-Roch theorem implies $A_2$ is equivalent to an effective divisor on $\Gamma_2$, since $\deg(A_2) = 1 = \operatorname{genus}(\Gamma_2)$. Thus, we conclude $D+(p)$ has rank at least $r$ on $\Gamma$.

We may identify $W^r_d (\Gamma_1)$ with a polyhedral subset of $W^r_d (\Gamma)$ and $W^0_1 (\Gamma_2)$ with a polyhedral subset of $W^{0}_1 (\Gamma)$ via the natural inclusion. Then, working inside $\Pic^{d+1} (\Gamma)$, the above argument furnishes an injective map from the Minkowski sum $W^r_d (\Gamma_1) + W^{0}_1 (\Gamma_2)$ into $W^r_{d+1} (\Gamma)$. It follows that $$\dim(W^r_{d+1} (\Gamma)) \ge \dim (W^r_d (\Gamma_1) + W^{0}_1 (\Gamma_2)) = \dim(W^r_d (\Gamma_1) ) + 1$$ where the second equality follows from the fact that images under the Abel-Jacobi map of the edges of $\Gamma_1$ are orthogonal to the image of $\Gamma_2$ with respect to the ``cycle intersection'' bilinear form, since the intersection of $\Gamma_2$ and any subset of $\Gamma_1$ is at most a point. 
\end{proof}

This theorem and technique is all that is necessary to prove Theorem A, which we restate and prove.

\begin{theorem}
A tree of loops $\Gamma$ is geometric Brill-Noether general if and only if $\Gamma$ is a path of loops.
\end{theorem}
\begin{proof}
If $\Gamma$ is a path of loops, then $\Gamma$ is geometric Brill-Noether general from the main result of \cite{cdpr}. Now suppose that $\Gamma$ is a tree of loops, but not a path of loops; we will show $\Gamma$ cannot be geometric Brill-Noether general. Since $\Gamma$ is not a path, it must contain a subgraph of genus $4$ that is a tree of loops but not a path of loops, denoted $\Gamma_0 \subseteq \Gamma$. Trees are connected, so we can construct $\Gamma$ by adding one loop at a time to $\Gamma_0$. Let $\Gamma_i$ be the $i^{\text{th}}$ step in this process, once $i$ cycles have been added to $\Gamma_0$. Then $\Gamma_i$ has genus $i+4$ for all $i$, and if $\Gamma$ has genus $g$, then $\Gamma = \Gamma_{g-4}$. We will prove by induction that $\Gamma_i$ is not geometric Brill-Noether general for all $i$, and thus that $\Gamma$ is not geometric Brill-Noether general.

As a base case, $\Gamma_0$ is not geometric Brill-Noether general, since by Lemma \ref{w13_mismatch}, $\dim(W^1_3(\Gamma_0)) \ge 1 > \rho(4,1,3) = 0$. Now assume that $\Gamma_i$ is not geometric Brill-Noether general, i.e. that there exist $r,d$ such that $\dim(W^r_d(\Gamma_i)) > \rho(i + 4, r, d)$. By Theorem \ref{dim_inequality}, we see
\begin{align*}
\dim (W^{r}_{d+1}(\Gamma_{i+1})) &\ge \dim (W^r_d(\Gamma_i)) + 1 \\
&\ge \rho(i + 4,r,d)+2 \\
&= \rho(i + 5,r,d+1) + 1 \\
&> \rho(i + 5,r,d+1).
\end{align*}

Thus, $\Gamma_{i+1}$ is not geometric Brill-Noether general, so by induction, $\Gamma$ is not geometric Brill-Noether general.
\end{proof}

\section{Weakly Geometric Brill-Noether generality}

Our second notion of Brill-Noether generality is a weakening of geometric Brill-Noether generality. One might wonder if $\rho$ is at minimum an indicator on trees of loops $\Gamma$ of whether the set $W^r_d(\Gamma)$ is empty.

\begin{definition}
A metric graph $\Gamma$ is \emph{weakly geometric Brill-Noether general} if $W^r_d(\Gamma)$ is nonempty whenever $\rho \ge 0$ and empty otherwise.
\end{definition}

Since this condition is strictly weaker than that of geometric Brill-Noether generality, 
it is conceivable that it all trees of loops would be weakly geometric Brill-Noether general. Nevertheless, most trees of loops still do not satisfy this condition.

\begin{proposition}
\label{2offnotweakly}
Let $\Gamma$ be a tree of loops of genus $g$ such that the longest path of loops consists of at most $g-2$ loops. Then $\Gamma$ is not weakly geometric Brill-Noether general.
\end{proposition}

\begin{proof}

We will consider two cases. Let $l$ be the number of loops in the longest path of loops in $\Gamma$. First, consider the case where $l$ is even. In this case, let $q$ be the intersection point of the middle two loops of a path of length $l$, as pictured. 

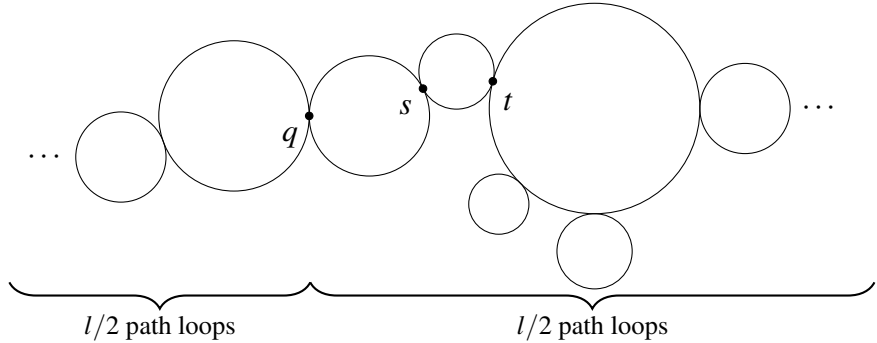
\begin{center}
\begin{tikzpicture}
\draw (-1,0) circle [radius=1];
\draw (0.8,0) circle [radius = 0.8];
\draw [fill] (0,0) circle [radius=0.05] node[below left] {$q$};
\draw [fill] ($(0.8,0) + (27:0.8)$)circle [radius=0.05] node[below left] {$s$};
\draw [fill] ($(0.8,0) + (27:1.3) + (345:0.5)$) circle [radius = 0.05] node [below right] {$t$};
\draw ($(0.8,0) + (27:1.3)$) circle [radius = 0.5];
\draw ($(0.8,0) + (27:1.3) + (345:1.9)$) circle [radius = 1.4];
\draw ($(0.8,0) + (27:1.3) + (345:1.9) + (225:1.8)$) circle [radius = 0.4];
\draw ($(0.8,0) + (27:1.3) + (345:1.9) + (270:1.9)$) circle [radius = 0.5];
\draw ($(0.8,0) + (27:1.3) + (345:1.9) + (2,0)$) circle [radius = 0.6];
\node (dots) at ($(0.8,0) + (27:1.3) + (345:1.9) + (3,0)$) {$\cdots$};
\draw ($(-1,0) + (200:1.6)$) circle [radius=0.6];
\node (ldots) at  ($(-1,0) + (200:1.6) + (-1,0)$) {$\cdots$};
\draw [thick, decorate,decoration={brace,amplitude=10pt,mirror},xshift=0.4pt,yshift=-0.4pt](0,-2.2) -- (7.5,-2.2) node[black,midway,yshift=-0.6cm] {\footnotesize $l/2$ path loops};
\draw [thick, decorate,decoration={brace,amplitude=10pt,mirror},xshift=0.4pt,yshift=-0.4pt](-4,-2.2) -- (0,-2.2) node[black,midway,yshift=-0.6cm] {\footnotesize $l/2$ path loops};
\end{tikzpicture}
\end{center}

Let $D = \left(\frac l2 + 1 \right)(q)$. By the tropical Riemann-Roch theorem, the rank of $D$ when restricted to either of the loops containing $q$ is $\frac l2$. In particular, for $s$ the connection point as in the diagram, $D - \frac l2 (s)$ is equivalent to an effective divisor. In other words, $D$ is equivalent to some divisor $D'$ with $\frac l2$ chips placed on $s$. The rank of $D'$ when restricted to one of the loops containing $s$ is then $\frac l2 - 1$, so $D'$ has rank $\frac l2 - 1$ when restricted to one of these loops. Thus $D' - \left(\frac l2 - 1\right)(t)$ is equivalent to an effective divisor, since chip-firing can be performed on each loop in isolation, so $D'$ is equivalent to some effective divisor $D''$ with $\frac l2 - 1$ chips placed on $t$.

We will show by induction on $k$ that for $L$ any loop at distance $k \le \frac l2$ from one of the two loops containing $q$, $D$ is equivalent to some effective divisor with degree at least $\frac l2 - k + 1$ when restricted to the loop $L$. The base case is precisely the argument that $\frac l2$ chips may be placed on the point $s$ above. For the inductive step, let $L$ be a loop at distance $k$ from one of the loops containing $q$, and let $L'$ be the loop at distance $k-1$ that is adjacent to $L$. Let $v$ be the intersection point between $L'$ and $L$. By the inductive hypothesis, $D$ is equivalent to an effective divisor $D^{(k-1)}$ with degree at least $\frac l2 - k + 2$ when restricted to $L'$. Since $v$ is in both loops $L$ and $L'$, it suffices to show that $D^{(k-1)}$ is equivalent to some divisor with $\frac l2 - k + 1$ chips placed on $v$, or that $D^{(k-1)} - \left( \frac l2 - k + 1 \right)(v)$ is an effective divisor. Since $k \le \frac l2$, $\deg(D^{(k-1)}) \ge 2$, so the rank of $D^{(k-1)}$ is 
\[r(D^{(k-1)}) = \deg(D^{(k-1)}) - 1 = \frac l2 - k + 1.\]
But then $D^{(k-1)} - \left( \frac l2 - k + 1 \right)(v)$ is equivalent to an effective divisor, just as desired. This completes the inductive argument.

Any loop in $\Gamma$ is connected to a loop containing $q$ via a path of loops of length at most $\frac l2 - 1$, since otherwise we would have a longer maximal path. Let $v$ be an arbitrary point in $\Gamma$. Let $L$ be the closest loop to $q$ containing $v$ and let $k \le \frac l2 - 1$ be the distance between $L$ and a loop containing $q$. Then by our inductive argument, $D$ is equivalent to an effective divisor $D^{(L)}$ with at least $\frac l2 - k + 1$ chips placed on $L$. The bound $k \le \frac l2 -1$ implies that $D^{(L)}$ has at least $2$ chips placed on $L$ itself. But then by the tropical Riemann-Roch Theorem, $D^{(L)}$ has rank $1$ when restricted to $L$, so $D^{(L)} - (v)$ is linearly equivalent to an effective divisor. Since $D^{(L)}$ and $D$ are linearly equivalent, $D - (v)$ must also be linearly equivalent to an effective divisor. Thus $D$ has rank $1$, so
\[W^1_{\frac l2 + 1}(\Gamma) \ne \emptyset.\]
However, since $\Gamma$ has genus $g$ where $l \le g -2$,
\begin{align*}
\rho\left(g, 1, \frac l2 + 1\right) &= g - 2 \cdot \left(g - \frac l2 - 1 + 1\right) \\
&= -g + l \\
&\le -2 < 0.
\end{align*}
Thus $\Gamma$ is not Brill-Noether general.

Now we consider the case where $l$ is odd; it is very similar. In this case, let $q$ be any point on the middle loop of the longest path. Let $D = \frac{l+3}2(q)$. By an inductive argument identical to that of the even case, if $L$ is any loop at distance $k \le \frac{l-1}2$ from the center loop, then $D$ is equivalent to some effective divisor with degree at least $\frac {l+3}2 - k$ when restricted to $L$. All loops are at distance at most $\frac{l-1}2$ from the center loop, so for any loop $L$, $D$ is equivalent to some effective divisor $D^{(L)}$ with degree at least $\frac{l+3}2 - \frac{l-1}2 = 2$ when restricted to $L$. Then for any point $v \in L$, $D^{(L)} - (v)$ is linearly equivalent to some effective divisor just as above, so $D - (v)$ must also be linearly equivalent to an effective divisor. Thus $D$ has rank $1$, so 
\[W^1_{\frac{l+3}2}(\Gamma) \ne \emptyset.\]
However, $\Gamma$ has genus $g$, with $l \le g - 2$, so
\begin{align*}
\rho\left(g,1,\frac{l+3}2 \right) &= g - 2 \cdot \left( g - \frac{l+3}2 + 1 \right) \\
&= g - 2g + l + 3 - 2 \\
&= -g + l + 1 \\
&\le -1 < 0.
\end{align*}
Thus $\Gamma$ is not Brill-Noether general.
\end{proof}

However, note that this argument proves that for negative values of $\rho$, the set $W^r_d(\Gamma)$ may still be nonempty. This is in fact the only way in which weakly geometric Brill-Noether generality may be violated. In \cite{v}, van der Pol proved that for trees of loops $\Gamma$, referred to in his paper as \emph{cactus graphs}, whenever the value of $\rho$ is positive, the set $W^r_d(\Gamma)$ is nonempty.

\section{Rank Brill-Noether generality}

Following the work done in \cite{lpp}, we examine a more combinatorial definition of rank.

\begin{definition} 
Suppose $\Gamma$ is a metric graph and $r,d$ are natural numbers such that $W^r_d (\Gamma)$ is nonempty. The \textit{Brill-Noether rank} $w^r_d (\Gamma)$ is the largest integer $k$ such that for every effective divisor $E$ of degree $r+k$, there exists $D$ of rank at least $r$ and degree $d$ such that $D-E$ is effective. If $W^r_d (\Gamma)$ is empty, then we set $w^r_d (\Gamma) = -1$.
\end{definition}

\begin{definition}
A metric graph $\Gamma$ of genus $g$ is \emph{rank Brill-Noether general} if:
\begin{enumerate}
\item $w^r_d(\Gamma) = -1$ whenever $\rho(g,r,d) < 0$.
\item $w^r_d(\Gamma) = \rho(g,r,d)$ whenever $0 \le \rho(g,r,d) \le g$.
\end{enumerate}
\end{definition}

The negative result obtained in examining weak Brill-Noether generality carries over, so certainly rank Brill-Noether generality is out of reach for most trees of loops.

\begin{theorem}
Let $\Gamma$ be a tree of loops of genus $g$ such that the longest path of loops consists of at most $g-2$ loops. Then $\Gamma$ is not rank Brill-Noether general.
\end{theorem}

\begin{proof}
This follows immediately from Proposition \ref{2offnotweakly}. If for some $r,d$, the set $W^r_d(\Gamma) \neq \emptyset$, then $w^r_d(\Gamma) \ge 0$. If $\Gamma$ has genus $g$ where the longest path of loops consists of at most $g-2$ loops, then $\Gamma$ is not weakly geometric Brill-Noether general and in particular there exist $r,d$ such that $W^r_d(\Gamma) \neq \emptyset$ despite $\rho(g,r,d)$ being negative. Thus $w^r_d(\Gamma) \ge 0$ despite $\rho(g,r,d)$ being negative, so $\Gamma$ is not rank Brill-Noether general.
\end{proof}

We do not know the relationship between $\rho(g,r,d)$ and $w^r_d$ for a general tree of loops $\Gamma$ when $\rho > 0$. However, there is a constraint relating $w^r_d$ of a metric graph and $w^r_{d+1}$ of the graph formed from gluing a loop onto that metric graph. This constraint may prove useful in further explorations.

\begin{theorem} 
Let $\Gamma_1$ be any metric graph, $\Gamma_2$ a loop, and $\Gamma = \Gamma_1 \wedge \Gamma_2$ an arbitrary wedge sum (with wedge point $q$). If $W^{r}_{d} (\Gamma_1)$ is nonempty, then 
\[w^{r}_{d} (\Gamma_1) \le w^{r}_{d+1} (\Gamma) \le w^{r-1}_{d} (\Gamma_1).\]
\end{theorem}

\begin{proof} 
Set $k = w^r_d(\Gamma_1)$; $k \ge 0$, since $W^r_d(\Gamma_1)$ is nonempty. 

First, we will prove the lower bound. Let $E \in \Div^{r+k}(\Gamma)$ be an arbitrary effective divisor of degree $r+k$; we will construct a divisor $D \in W^r_{d+1}(\Gamma)$ with $D-E$ effective. We may first $q$-reduce $E$ to obtain $E' \sim E$, where $E'$ must still be effective and must contain either one chip or zero chips placed on $\Gamma_2 \setminus \{q\}$. Note that if there exists a divisor $D'$ with $D'-E'$ effective, we can let $D = (D'-E') + E = D' + (E-E')$, which is then a divisor, still in $W^r_{d+1}(\Gamma)$, with $D-E = D'-E'$ effective. Thus it suffices to find a satisfactory divisor for the $q$-reduced case.

If $E'$ contains exactly one chip on a point $p \in \Gamma_2 \setminus \{q\}$, let $D_2$ be the divisor $(p)$ on $\Gamma_2$. $E'\vert_{\Gamma_1}$ is an effective divisor of degree $r+k-1$, so there is a divisor $D_1 \in W^r_d(\Gamma_1)$ with $D_1 - E'\vert_{\Gamma_1}$ effective. Then $D' = D_1\vert^{\Gamma} + D_2\vert^{\Gamma}$ is an element of $W^r_{d+1}(\Gamma)$ by the argument in Theorem \ref{dim_inequality}, with $D'-E'$ effective; thus there is a divisor $D \in W^r_{d+1}(\Gamma)$ with $D - E$ effective.

Now assume that $E'$ contains no chips on $\Gamma_2 \setminus \{q\}$. Then $E'\vert_{\Gamma_1}$ has degree $r+k$, so we can pick a divisor $D_1 \in W^r_d(\Gamma_1)$ with $D_1 - E'\vert_{\Gamma_1}$ effective. Let $p$ be any point in $\Gamma_2$; Then the divisor $D' = D_1\vert^{\Gamma} + (p)$ is an element of $W^r_{d+1}(\Gamma)$ with $D'-E'$ effective. Since $E \sim E'$, there exists a divisor $D \in W^r_{d+1}(\Gamma)$ with $D-E$ effective as well, so $w^r_{d+1}(\Gamma) \geq k = w^r_d(\Gamma_1)$.

We will now prove the upper bound. Let $l = w^r_{d+1}(\Gamma)$. We will show that $l$ is a lower bound for $w^{r-1}_d(\Gamma_1)$, or equivalently, that for any effective divisor $E_1$ on $\Gamma_1$ of degree $r-1 + l$, there exists a divisor $D_1 \in W^{r-1}_d(\Gamma_1)$ with $D_1 - E_1 \ge 0$.

Fix any effective divisor $E_1$ on $\Gamma_1$ of degree $r-1 + l$. Just as above, we may restrict to addressing the case when $E_1$ is $q$-reduced, so assume that this is the case. Let $E = (E_1 + (q))|^{\Gamma}$, which is then an effective divisor of degree $r + l$ on $\Gamma$. Note that $E$ remains $q$-reduced. Since $w^r_{d+1}(\Gamma) = l$ and $E$ is $q$-reduced, there exists a $q$-reduced divisor $D$ of degree $d+1$ on $\Gamma$ with $D - E \ge 0$.

Since $D$ is a $q$-reduced divisor on $\Gamma$, $D$ either has one chip on the loop $\Gamma \setminus \Gamma_1 = \Gamma_2 \setminus \{q\}$, or all chips are on $\Gamma_1$. If $D$ has one chip on $\Gamma_2 \setminus \{q\}$, let $p$ be the point at which this chip is located. If not, let $p = q$. 

Then $D - (p)$ is a divisor with no chips on $\Gamma \setminus \Gamma_1$. Since $D$ has rank at least $r$, as an element of $W^r_{d+1}(\Gamma)$, $D-(p)$ has rank at least $r-1$, since any additional $r-1$ chips may be removed from $D$ after the removal of $(p)$, and the result will still be equivalent to an effective divisor. Thus $D \in W^{r-1}_d(\Gamma)$; since $D-(p)$ has support on $\Gamma_1$, $D_1 = D-(p) \in W^{r-1}_d(\Gamma_1)$, as well.

It suffices to show that $D_1 = D- (p)$ satisfies $D_1 - E_1 \ge 0$; then for every choice of $E_1$ we will have found a divisor $D_1 \in W^{r-1}_d(\Gamma_1)$ with $D_1 - E_1 \ge 0$. If $p = q$, then
\begin{align*}
D_1 - E_1 &= D - (q) - (E - (q)) \text{, since $E = E_1 + (q)$} \\
&= D - E - (q) + (q) \\
&= D - E \ge 0\text{, just as desired.}
\end{align*}
Now if $p \ne q$, we know that $D|_{\Gamma_1} = D_1 = D - (p)$. Since $E = E_1 + (q)$ lies entirely on $\Gamma_1$, $D|_{\Gamma_1} - E \ge 0$, since $D - E \ge 0$. But then
\begin{align*}
D|_{\Gamma_1} - E &\ge 0 \\
\Rightarrow D_1 - E &\ge 0 \\
\Rightarrow D_1 - E_1 - (q) &\ge 0 \\
\Rightarrow D_1 - E_1 &\ge 0\text{, just as desired.}
\end{align*}
Thus we have proven that $D_1 - E-1 \ge 0$, so $w^{r-1}_d(\Gamma_1) \ge w^r_{d+1}(\Gamma)$, which is exactly the desired result.
\end{proof}

\section{Acknowledgments}

This research was supervised by Farbod Shokrieh, at the Cornell Summer Program for Undergraduate Research in 2015. We would like to thank Farbod Shokrieh for his guidance and supervision in writing this paper, the other participants and advisors in the Cornell SPUR program for their helpful discussions, comments, and suggestions, and Brian Osserman for his comments and interest.


\providecommand{\bysame}{\leavevmode\hbox to3em{\hrulefill}\thinspace}
\providecommand{\MR}{\relax\ifhmode\unskip\space\fi MR }
\providecommand{\MRhref}[2]{%
  \href{http://www.ams.org/mathscinet-getitem?mr=#1}{#2}
}
\providecommand{\href}[2]{#2}

\vspace*{5mm}

\end{document}